\documentclass{amsart}

\usepackage{amssymb}

    \newtheorem{theorem}{Theorem}[section]
    \newtheorem{corollary}[theorem]{Corollary}
    \newtheorem{definition}[theorem]{Definition}
    \newtheorem{lemma}[theorem]{Lemma}
    \newtheorem{proposition}[theorem]{Proposition}
    \newtheorem{remark}[theorem]{Remark}

    \def\cocoa{{\hbox{\rm C\kern-.13em o\kern-.07em C\kern-.13em o\kern-.15em A}}}
    
    \def\mz{{\mathbb{Z}}}

    \def\mr{{\mathbb{R}}}

    \def\mb{{\mathbb{B}}}
    \def \bx{{\pmb x}}

    \def\b0{{\pmb 0}}

    \def\rma{{\mathbb{T}}}
    \def\zma{\mathbb{Z}_{max}}
    
    \def\smf{A} 
    \def\smfn{\smf[x_1, \dots, x_n]} 
    \def\smfnL{\smf(x_1, \dots, x_n)} 
    
    \def\fracr{\mbox{Frac}(R)}
    \def\fraca{\mbox{Frac}(A)}
    \def\fracap{\mbox{Frac}(A')}
    \def\otherc{\mathfrak{C}}

	\def\dsim{\diamond}

    \def\diag{{\Delta}}

    \def\<{\langle}
    \def\>{\rangle}

\begin{document}

    \title{On the dimension of polynomial semirings}
    
    \author{D\'aniel Jo\'o}
    \address{R\'enyi Institute of Mathematics, Hungarian Academy of Sciences, Budapest, Hungary}
    \email{joo.daniel@renyi.mta.hu}
    \thanks{The first author is partially supported by OTKA K101515 and the exchange program between the Bulgarian and Hungarian Academies of Sciences.}

    \author{Kalina Mincheva}
    \address{Department of Mathematics, Johns Hopkins University, Baltimore, MD 21218}
    \email{mincheva@math.jhu.edu}



\begin{abstract}
In our previous work, motivated by the study of tropical polynomials, a definition for prime congruences was given for an arbitrary commutative semiring. It was shown that for additively idempotent semirings this class exhibits some analogous properties to prime ideals in ring theory. The current paper focuses on the resulting notion of Krull dimension, which is defined as the length of the longest chain of prime congruences. Our main result states that for any additively idempotent semiring $A$, the semiring of polynomials $A[x]$ and the semiring of Laurent polynomials $A(x)$, we have $\dim A[x] = \dim A(x) = \dim A + 1$.

\end{abstract}

\maketitle

\section{Introduction}

The current work studies the Krull dimension of additively idempotent semirings defined in terms of congruences. One motivation to study additively idempotent semirings is provided by tropical geometry. There, from an algebraic point of view, one is interested in the properties of polynomial rings over the tropical max-plus semifield $\rma = \mr_{max}$. Two other additively idempotent semifields are central to the development of characteristic 1 geometry \cite{CC13}. These are the semifield of integers $\zma \subset \rma$ and the two element additively idempotent semifield $\mb$.\par\smallskip
In ring theory congruences, or equivalently homomorphisms, are determined by the ideal that is their kernel (i.e. the equivalence class of the $0$ element). From this perspective semirings behave quite differently: in general the kernel of a congruence contains very little information about which elements are identified. In fact one can easily find examples of commutative semirings with a complicated lattice of congruences all of which have trivial kernels. With this consideration in mind, in \cite{JM14} the notion of primeness is extended to the congruences of a general commutative semiring, and the class of prime congruences is described in the polynomial (and Laurent polynomial) semirings over the semifields $\mb$, $\rma$ and $\zma$. The key application of this theory in \cite{JM14} is to prove a tropical Nullstellensatz. However, it is also observed that the prime congruences yield a notion of Krull dimension which behaves intuitively in the sense that an $n$ variable polynomial ring over $\rma,\zma$ or $\mb$ always has dimension $n$ larger than that of its ground semifield. The aim of the current paper is to generalize this result for the polynomial and Laurent polynomial semirings over arbitrary additively idempotent commutative semirings, which we will refer to as $\mb$-algebras.\par\smallskip
A different approach to establish a tropical notion of dimension using chains of congruences was taken by L. Rowen and T. Perri in \cite{PR15}, which we will briefly recall in Remark \ref{rem:perri-rowen}.  A common theme in \cite{PR15} and the present work is that to obtain a good notion of dimension one has to bypass the difficulties that come from polynomial semirings having too many congruences. In fact, as we will see in Proposition \ref{prop:infascqc}, any additively idempotent polynomial semiring in at least $2$ variables has infinite chains of congruences with cancellative quotients.\par\smallskip
The main result of the current work is Theorem \ref{thm:dplusone}, which concerns the polynomial semiring $A[x]$ and the Laurent polynomial semiring $A(x)$ over an arbitrary additively idempotent commutative semiring $A$. 
\begin{theorem}
Let $A$ be a $\mb$-algebra with $\dim A < \infty$. Then we have that $\dim A(x) = \dim A[x] = \dim A + 1$.
\end{theorem}

When comparing to the classical ring theoretic setting, it is somewhat surprising that Theorem \ref{thm:dplusone} holds without any restriction on $A$. 
In commutative ring theory, for any Noetherian ring $R$ $\dim R[x] = \dim R + 1$ (see for example \cite{Ei95}). However, when the Noetherian restriction is dropped $\dim R[x]$ can be any integer between $\dim R + 1$ and $2\dim R + 1$ (see \cite{Se54}).
\par\smallskip
The proof of Theorem \ref{thm:dplusone} relies on two key facts. The first one is Proposition \ref{prop:trivkerchain}. There we establish that the dimension of a domain (a $\mb$-algebra whose trivial congruence is prime) always equals the dimension of its semifield of fractions. This counter-intuitive fact can be explained by realizing that semifields are not the simple objects amongst semirings, but come with the distinguished property that the kernel of every congruence is trivial. One can then go on to show that the dimension of domains can be always computed by considering primes with trivial kernel. \par 
The second key observation is (i) of Proposition \ref{prop:laurentdim}:

\begin{proposition}
If ${\mathfrak{p}_1} \subset {\mathfrak{p}_2} \subset \dots$ is a chain of primes in $\smf(x)$ or $\smf[x]$ such that the kernel of every $\mathfrak{p}_i$ is the same, then after restricting the chain to $A$, in ${\mathfrak{p}_1}|_{\smf} \subseteq {\mathfrak{p}_2}|_{\smf}\dots$ equality occurs at most once.
\end{proposition}

It is noteworthy that this statement alone implies Theorem \ref{thm:dplusone} for the special case of Laurent semirings over semifields.\par
Section \ref{sec:prelim} contains the preliminaries, including some of the results of \cite{JM14}. Section \ref{sec:dimension} contains our main result and its proof. 

\par
\section*{Acknowledgements} The authors would like to thank M\'arton Hablicsek and M\'at\'e Lehel Juh\'asz for the inspiring discussions on the topic.

    \section{Preliminaries}\label{sec:prelim}
    
    In this section we recall some definitions and results from \cite{JM14}, which we refer to in the following chapters.
    \par\smallskip
    In this paper by a {\it semiring} we mean a commutative semiring with multiplicative unit, that is a nonempty set $R$ with two binary operations $(+,\cdot)$ such that $R$ is a commutative monoid with respect to both operations, multiplication distributes over addition and multiplication by $0$ annihilates $R$. A {\it semifield} is a semiring in which all nonzero elements have multiplicative inverse. 
    
    We recall the definition of the three semifields that play a central role in \cite{JM14}. We denote by $\mb$ the semifield with two elements $\{1,0\}$, where $1$ is the multiplicative identity, $0$ is the additive identity and $1+1 = 1$. The {\it tropical semifield}  $\rma$ is defined on the set $\{-\infty\} \cup \mr$, by setting the $+$ operation to be the usual maximum and the $\cdot$ operation to be the usual addition, with $-\infty$ playing the role of the $0$ element. Finally the semifield $\zma$ is the subsemifield of integers in $\rma$.
    
    A polynomial (resp. Laurent polynomial) semirings in $n$ variables $\bx = (x_1,\dots,x_n)$ over a semiring $R$ are defined in the usual way and denoted by $R[\bx]$ (resp. $R(\bx)$).
    \begin{definition}{\rm
    A {\it congruence} $C$ of the semiring $R$ is an equivalence relation on $R$ compatible with the semiring structure.    }
    \end{definition}

    The unique smallest congruence is the diagonal of $R \times R$ which is denoted by $\diag$, also called the {\it trivial congruence}. 
    $R \times R$ itself is the {\it improper congruence} the rest of the congruences are called {\it proper}. Quotients by congruences can be considered in the usual sense, the quotient semiring of $R$ by the congruence $C$ is denoted by $R/C$. 
    
    The {\it kernel} of a congruence is just the equivalence class of the $0$ element. We say that the kernel of a congruence is trivial if it equals $\{0\}$. \par\smallskip 
    
    By the {\it kernel of a morphism} $\varphi$ we mean the preimage of the trivial congruence $\varphi^{-1}(\diag)$, it will be denoted by $Ker(\varphi)$. If $R_1$ is a subsemiring of $R_2$ then the restriction of a congruence $C$ of $R_2$ to $R_1$ is $C|_{R_1}=C\cap (R_1 \times R_1)$.\par\smallskip

   A congruence is called {\it irreducible}  if it can not be obtained as the intersection of two strictly larger congruences. 
    
    A semiring is called {\it cancellative} if whenever $ab = ac$ for some $a,b,c \in R$ then either $a = 0$ or $b=c$. A congruence $C$ for which $R/C$ is cancellative will be called {\it quotient cancellative} or {\it QC}.
    
    Elements of $R \times R$ will be called {\it pairs}, and the smallest congruence containing the pair $\alpha \in R \times R$ will be denoted by $\<\alpha\>_R$ or $\<\alpha\>$ when there is no ambiguity. 
    The {\it twisted product} of the pairs $\alpha = (\alpha_1, \alpha_2)$ and $\beta = (\beta_1, \beta_2)$ is $(\alpha_1\beta_1+\alpha_2\beta_2,\alpha_1\beta_2+\alpha_2\beta_1)$. Whenever for some pairs $\alpha$ and $\beta$ we write $\alpha\beta$ we always mean this twisted product, and similarly, $\alpha^n$ always denotes the twisted $n$-th power of the pair $\alpha$.
    
    
    \par\smallskip
     We call {\it $\mb$-algebra} a commutative semiring with idempotent addition. Throughout this section $A$ denotes an arbitrary $\mb$-algebra. The idempotent addition defines an ordering: $a \geq b \iff a+b = a.$ We call a $\mb$-algebra totally ordered if the idempotent addition defines a total ordering on its elements, i.e. for every $a,b \in A$ one has $a+b = a$ or $a+b = b$.

    \begin{proposition}\label{prop: congbasic} (Proposition 2.2, \cite{JM14})
    Let $C$ be a congruence of a $\mb$-algebra $A$,
    \begin{itemize}
    \item[(i)] For $\alpha \in C$ and an arbitrary pair $\beta$ we have $\alpha\beta \in C$.
    \item[(ii)] If $(a,b) \in C$ and $a \leq c \leq b$ then $(a,c) \in C$ and $(b,c) \in C$. In particular if $(a,0) \in I$ then for every $a\geq c$ we have $(c,0) \in C$.
    \end{itemize}
    \end{proposition}
    
    Proposition \ref{prop: congbasic} has the following important consequence:

\begin{proposition}\label{prop:smfnokernel}
 If $F$ is an additively idempotent semifield then every proper congruence in the semiring of Laurent polynomials $F(x_1,\dots,x_n)$ has a trivial kernel.
\end{proposition}
\begin{proof}
If $f  \in F(x_1,\dots,x_n)$ is in the kernel of a proper congruence $I$ then by (ii) of  Proposition \ref{prop: congbasic} we also have that every monomial that appears in $f$ is in the kernel of $I$. On the other hand every monomial in a Laurent semiring over a semifield has multiplicative inverse. Hence if a monomial is in the kernel of a congruence $I$ then so is the multiplicative identity of $F(x_1,\dots,x_n)$, which implies that $I$ is improper. 
\end{proof}
    
    \begin{definition}{\rm We call a congruence $P$ of a  semiring $R$  {\it prime} if it is proper and for every $\alpha, \beta \in R \times R$ such that $\alpha\beta \in P$ either $\alpha \in P$ or $\beta \in P$. We call a semiring a domain if its trivial congruence is prime.
    }
    \end{definition}

\begin{remark}\label{rem:primeandQC}{\rm
The heuristics for choosing this definition is that for a commutative ring $R$ a congruence $C \subset R \times R$ is prime in the above sense if and only if its kernel is a prime ideal in the usual sense. It is also easy to deduce from the definition that every prime congruence is QC (or equivalently every domain is cancellative) and irreducible. The converse is also true - but not obvious: in Theorem 2.12 of \cite{JM14} it was shown that a congruence of a $\mb$-algebra is prime if and only if it is QC and irreducible. The key difference from ring theory (where the class of QC and prime congruences coincide) is that a QC congruence does not need to be irreducible and - as we will see at the end of this section - there are typically much more QC congruences than primes.}
\end{remark}


We recall the following characterization of $\mb$-algebras that are domains:

\begin{proposition}\label{prop:primeorder}(Proposition 2.9, \cite{JM14})
A $\mb$-algebra $A$ is a domain if and only if it is cancellative and totally ordered.  
\end{proposition}
    
    We define dimension similarly to the Krull-dimension in ring theory:
    \begin{definition}{\rm
    The {\it dimension} of a $\mb$-algebra $A$ is the length of the longest chain of prime congruences in $A \times A$ (where by length we mean the number of strict inclusions). The dimension of $A$ will be denoted by $\dim A$.
    }
    \end{definition}
    
While in ring theory every field has Krull-dimension 0, it is not the case for semifields, for example the reader can easily check that $\dim\zma = \dim\rma = 1$.

\begin{proposition}\label{prop: dimB}(Proposition 2.5, \cite{JM14})
\begin{itemize}
\item[(i)] Every $\mb$-algebra maps surjectively onto $\mb$. 
\item[(ii)] The only $\mb$-algebra that is a domain and has dimension $0$ is $\mb$.
\end{itemize}
 \end{proposition}

We point out that in Proposition 2.5 of \cite{JM14} only the part (i) of the statement was made explicit but (ii) follows immediately. In Theorems 4.9, 4.10 and 4.14 of \cite{JM14} a description of the primes of the polynomial and Laurent polynomial rings over $\mb$, $\zma$ and $\rma$ which was then used to calculate the dimensions in each of this cases.

\begin{proposition}\label{prop:jm14dims}
$\ $
\begin{itemize}
\item[(i)] $\dim\mb[x_1,\dots,x_n] = \dim\mb(x_1,\dots,x_n) = n$.
\item[(ii)] $\dim \zma[x_1,\dots,x_n] = \dim \zma(x_1,\dots,x_n) = n+1$.
\item[(iii)] $\dim \rma[x_1,\dots,x_n] = \dim \rma(x_1,\dots,x_n) = n+1$.

\end{itemize}
\end{proposition}
 
\begin{remark}\label{rem:perri-rowen}
{\rm{
In \cite{PR15} the authors consider a sublattice of all congruences of a rational function semifield, generated by the so-called hyperplane kernels. They define the dimension as the maximum of the length of chains of irreducibles in this particular sublattice. It is then verified, amongst several other results, that the dimension of a rational function field over an archimedean semifield equals the number of variables. This result is somewhat analogous to the main theorem of the current paper, however neither of the two results imply special cases of the other, since our notion of dimension differs from that of \cite{PR15}. We also note that the results of \cite{PR15} are set in the more general context of "supertropical algebra", but this setting contains the usual semirings as a degenerate special case. To avoid possible confusion we point out that our terminology differs from that of \cite{PR15}, where the authors call every cancellative semiring a domain. Also in \cite{PR15} kernels refer to the equivalence class of $1$ in a congruence (of a semifield) and not to the equivalence class of $0$ as in the current paper. }
}
\end{remark}
  
We mentioned in Remark \ref{rem:primeandQC} that QC congruences do not need to be irreducible. Indeed one can find several examples of such congruences by considering the following proposition:

\begin{proposition}\label{prop:intersectprimes}
Let $P_i$ denote the elements of a (possibly infinite) set of prime congruences with trivial kernels in an $\mb$-algebra A. Then $\bigcap P_i$ is a QC congruence.
\end{proposition}
\begin{proof}
Assume $(xa,xb) \in \bigcap P_i$ for some $x,a,b \in A$ and $x \neq 0$. Then $(xa,xb) = (x,0)(a,b) \in P_i$ for every $i$. By the assumptions $(x,0) \notin P_i$ for any $i$, hence the prime property implies that $(a,b) \in \bigcap P_i$. 
\end{proof}

Finally we show that the two variable polynomial (or Laurent polynomial) semiring over any $\mb$-algebra contains an infinite ascending chain of QC congruences, hence the class of QC congruences - without further restrictions - does not yield an interesting notion of Krull-dimension. 
    
\begin{proposition}\label{prop:infascqc}
For a $\mb$-algebra $A$ the semirings $A[x,y]$ and $A(x,y)$ contain infinite ascending chains of QC congruences.
\end{proposition}    
\begin{proof}
By Proposition \ref{prop: dimB}, $\mb$ is a quotient of $A$, hence it is enough to prove the statement for the case $A = \mb$. We recall from Section 4 of \cite{JM14} that to a non-zero real vector $v \in \mr^2$ one can assign a (minimal) prime $P_v$ in $\mb[x,y]$ or $\mb(x,y)$ which is generated by the set of pairs $$\{ (x^{n_1}y^{n_2} + x^{m_1}y^{m_2}, x^{n_1}y^{n_2} ) \mid\smallskip v_1 n_1 + v_2 n_2 \geq  v_1 m_1 + v_2 m_2\}.$$
In other words one takes a (possibly not complete) monomial order by scalar multiplying exponent vectors with a fixed $v$, and the congruence $P_v$ identifies each polynomial with its leading term.  Set $C_n = \bigcap_{k \geq n} P_{(k,1)}$. We claim that $C_1 \subset C_2 \subset \dots$ is an infinite ascending chain of congruences with cancellative quotients. Indeed they are QC by Proposition \ref{prop:intersectprimes} and are contained in each other by definition. Moreover the containments are strict since $(x+y^j, x) \in P_k$ 
 if and only if $k \geq j$.
\end{proof}

    Throughout this paper every semiring will be additively idempotent. We will denote by $\smf(\bx) = \smfnL$ and $\smf[\bx] = \smfn$, and we will use the shorter notation when this does not lead to ambiguity. By a prime we will always mean a prime congruence. A maximal chain of primes will be a (non-refinable) chain of maximal length.
    \newline

    \section{Dimension of the polynomial and Laurent polynomial semirings}\label{sec:dimension}

We begin by showing that the dimension of the polynomial or Laurent polynomial semirings over a finite dimensional $\mb$-algebra is strictly bigger than the dimension of the underlying $\mb$-algebra.

\begin{proposition}\label{prop: var_inc_dim}
Let $A$ be a $\mb$-algebra of finite Krull dimension, then $\dim A(x) \geq \dim A + 1$ and $\dim A[x] \geq \dim A + 1$.
\end{proposition}
    
    \begin{proof} First assume $A$ is a domain. By Proposition \ref{prop:primeorder} it is totally ordered with respect to the order coming from addition. Consider the following total ordering on the set of monomials of $A(x)$. Let $a_1y^{n_1}$ and $a_2y^{n_2}$ be two monomials, then $a_1y^{n_1} > a_2y^{n_2}$ if $n_1 > n_2$ or if $n_1 = n_2$ and $a_1 > a_2$. Since $A$ is a domain we can always compare the coefficients. This ordering is compatible with the multiplication on $A(x)$.

    Consider the congruence generated by $(b+c,c)$, when $c \geq b$, where $b,c$ are monomials of $A(x)$. Denote by $D$ the quotient of $A(x)$ by this congruence and let $$\phi: A(x) \rightarrow D,$$ be the quotient map. Note that $D$ is a domain by Proposition \ref{prop:primeorder} because it is totally ordered by construction and is cancellative. The kernel of $\phi$ is a prime congruence, hence $\dim A(x) \geq \dim D$. Now consider an evaluation morphism $$ \psi: D \rightarrow A, \ y \mapsto 1.$$  
Note that $D / \ker \psi = A$, hence $\ker \psi$ is a non-trivial prime congruence of $D$ and thus $\dim D > \dim A$. Hence $\dim A(x) \geq \dim A + 1$.\par

    If $A$ is not a domain, then consider a prime $\mathfrak{p}$ which is part of a maximal chain for $A$. Note that $A/\mathfrak{p}$ is a domain since $\mathfrak{p}$ is prime and $\dim A/\mathfrak{p} = \dim A$. Since $(A/\mathfrak{p})(x)$ is a quotient of $A(x)$ we have $\dim A(x) \geq \dim (A/\mathfrak{p})(x)$, thus $\dim A(x) \geq \dim A + 1$ follows from the first part of the proof. The proof for the case of the polynomial semiring $A[x]$ is essentially the same.

    \end{proof}

One can immediately obtain the following:

    \begin{proposition}\label{lemma: kalina_lemma}
    If $A$ is a $\mb$-algebra and $\dim A[x] = 2$ (or  $\dim A(x) = 2$) then 
    $\dim A = 1$.
    \end{proposition}

    \begin{proof}
    By Proposition  \ref{prop: var_inc_dim} $\dim A(x) > \dim A$ (resp. $\dim A[x] > \dim A$). Thus $\dim A = 0$ or $1$. If $\dim A = 0$ then by Proposition \ref{prop: dimB} $A/P =\mb$ for any prime $P$ of $A$. Hence any strictly increasing chain of primes in $A(x)$ maps to a strictly increasing chain of primes in $\mb(x)$, and by Proposition \ref{prop:jm14dims} we have $\dim A(x) = \dim\mb(x) = 1$.
    \end{proof}


Next, we show that chains of prime congruences of $A(x)$ in which all primes have the same kernel can stabilize at most once when restricted to $A$. We will need the following two simple lemmas:

    \begin{lemma}\label{lemma: calc} 
    Let $A$ be a cancellative $\mb$-algebra and $a,b,c,d \in A$ such that $a>b$ and $c>d$, then $ac > bd$. 
    \end{lemma}
    
    \begin{proof}
    Clearly $ac \geq ad \geq bd$. If $ac = bd$, then we have $ac = ad$, and then by
    cancellativity $c = d$ or $a = 0$ both contradicting our assumptions.
    \end{proof}
    
     \begin{lemma}\label{lemma: power} 
    Let $A$ be a $\mb$-algebra and $P$ be a prime congruence in $A\times A$. If $(x^n, y^n) \in P$ for $n > 0$ then $(x,y) \in P$.
    \end{lemma}
    
    \begin{proof} Consider $A/P$, which is a domain since $P$ is prime. Then we have that $x^n = y^n$ in $A/P$. We want to show that $x=y$. Assume for contradiction that $x \neq y$. Recall that domains are totally ordered so without loss of generality assume that $x > y$. Then after applying Lemma \ref{lemma: calc} $n$ times we arrive at a contradiction.
    \end{proof}

We are ready to prove:

\begin{lemma}\label{lem:fourprimes}
Let $\smf$ be a $\mb$-algebra and $P_1 \subsetneq P_2 \subseteq P_3 \subsetneq P_4$ prime congruences of $\smf(x)$ (resp. $\smf[x]$), satisfying $\ker(P_1) = \ker(P_2)=\ker(P_3)=\ker(P_4)$. Then at least one of $P_1|_\smf \subsetneq P_2|_\smf$ or $P_3|_\smf \subsetneq P_4|_\smf$ holds.
\end{lemma}
\begin{proof}

    Assume for a contradiction that $P_1|_\smf = P_2|_\smf$ and $P_3|_\smf = P_4|_\smf$. Since  $P_1 \subsetneq P_2$ and  $P_3 \subsetneq P_4$ there exist two pairs, 
    \begin{center}
    $(f_1, g_1) \in P_2 \setminus P_1$, for some $f_1, g_1 \in \smf(x)\ (resp.\ \smf[x])$\\
    $(f_2, g_2) \in P_4 \setminus P_3$, for some $f_2, g_2 \in \smf(x)\ (resp.\ \smf[x])$
    \end{center}
The quotient by a prime is totally ordered by Proposition \ref{prop:primeorder}, which by the definition of the ordering means that every sum is identified with at least one of its summands. Hence we may assume that $f_1, f_2, g_1$ and $g_2$ are monomials and write the following instead:
    \begin{center}
    $(ay^{k_1}, b y^{k_2}) \in P_2 \setminus P_1$, for some $a_1,b_1 \in \smf$\\
    $(cy^{m_1}, d y^{m_2}) \in P_4 \setminus P_3$, for some $a_2,b_2 \in \smf,$
    \end{center}
By the assumption that the kernels of $P_{1,2,3,4}$ are the same, none of the elements of the above pairs may be in $\ker(P_1)=\dots =\ker(P_4)$, implying that $a,b,c,d \notin \ker(P_1)$. It also follows that if $y \in \ker(P_1)$ then $k_1=k_2=m_1=m_2=0$ and the statement follows from $(a,b) \in  P_2 \setminus P_1$ and $(c,d) \in P_4 \setminus P_3$. For the remainder of the proof we assume that $y \notin \ker(P_1)$. Without loss of generality we can assume that $k_1 \geq k_2$ and $m_1 \geq m_2$, and set $k = k_1 - k_2$ and $m = m_1 - m_2$. Since the quotient by a prime is cancellative and $y$ is not in the kernel of any of $P_{1,2,3,4}$ it follows that $(ay^k, b)  \in P_2 \setminus P_1$ and $(cy^m, d)  \in P_4 \setminus P_3$. Also by the assumption $P_1|_\smf = P_2|_\smf$ and $P_3|_\smf = P_4|_\smf$, we have that $k,m > 0$.

 Thus we have, 
    $$(a^m y^{km}, b^m) \in P_2 \subset P_4$$
    $$(c^k y^{km}, d^k) \in P_4$$
 Multiplying the first equation with $c^k$ the second with $a^m$ we obtain:
    $$(b^m c^k, d^k a^m) \in P_4$$
    as ${P_3}|_{\smf} = {P_4}|_{\smf} $ we also have 
    \begin{equation*}\label{key_pair}(b^m c^k, d^k a^m) \in P_3\end{equation*}
    Multiplying by $y^{km}$
    $$(b^m c^k y^{km}, d^k a^m y^{km}) \in P_3$$
    But we also know that 
    $$(a^m y^{km}, b^m) \in P_2 \subseteq P_3$$
    So from the above two we obtain that 
    \begin{equation}\label{b_eq_0} (b^m c^k y^{km}, d^k b^m) \in P_3 \end{equation}
    Now since $b \notin \ker(P_3)$ we also have that $b^m \in \ker(P_3)$, since $P_3$ is prime implying that its quotient is cancellative. Thus we obtain:
    $$(c^k y^{km}, d^k) \in P_3$$
    But then, since $k>0$, by Lemma \ref{lemma: power} we have
    $$(cy^m, d) \in P_3$$
    a contradiction.

\end{proof}

\begin{proposition}\label{prop:laurentdim}
\begin{itemize}
\item[(i)] If ${\mathfrak{p}_1} \subset {\mathfrak{p}_2} \subset \dots$ is a chain of primes in $\smf(x)$ or $\smf[x]$ such that the kernel of every $\mathfrak{p}_i$ is the same, then after restricting the chain to $A$, in ${\mathfrak{p}_1}|_{\smf} \subseteq {\mathfrak{p}_2}|_{\smf}\dots$ equality occurs at most once.
\item[(ii)] For an additively idempotent semifield $F$ we have $\dim F(x_1,\dots,x_n) = \dim F + n$.
\end{itemize}
\end{proposition}
\begin{proof}
For (i), assume for contradiction that equality occurs at least twice, say ${\mathfrak{p}_i}|_{\smf} = {\mathfrak{p}_{i+1}}|_{\smf}$ and ${\mathfrak{p}_j}|_{\smf} = {\mathfrak{p}_{j+1}}|_{\smf}$ with $i+1 \leq j$. Then by setting $P_1 = {\mathfrak{p}_i}$, $P_2 = {\mathfrak{p}_{i+1}}$, $P_3 = {\mathfrak{p}_j}$ and $P_4 = {\mathfrak{p}_{j+1}}$ we arrive at contradiction with Lemma \ref{lem:fourprimes}. (ii) follows by induction from (i) and Proposition \ref{prop:smfnokernel} which asserts that in $F(x_1,\dots,x_n)$ the kernel of every congruence is trivial.
\end{proof}
   
We recall that a cancellative semiring $R$ embeds into its semifield of fractions $\fracr$. The elements of $\fracr$ are the equivalence classes in $R \times (R\setminus\{0\})$ of the relation $(r_1,s_1)\sim(r_2,s_2)\Leftrightarrow r_1s_2 = r_2s_1$, with operations $(r_1,s_1)+(r_2,s_2) = (r_1s_2+r_2s_1, s_1s_2)$, $(r_1,s_1)(r_2,s_2) = (r_1r_2,s_1s_2)$. As usual for $(r,s) \in \fracr$ we will write $\frac{r}{s}$. We refer to \cite{Go99} for the details of this construction.


We would like to point out that part (i) of Proposition \ref{prop:fracrcongs} is essentially the same as Lemma 2.4.4 of \cite{PR15} and both of parts (i) and (ii) are likely well-known. We provide a short proof for the convenience of the reader. Also, note that Proposition \ref{prop:fracrcongs} is not specific to the additively idempotent case.

\begin{lemma}\label{lem:smftransitive}
Let $F$ be a semifield. Let $C \subseteq F \times F$ be symmetric and reflexive and closed under addition and multiplication, that is for $(a_1,b_1),(a_2,b_2) \in C$ we have that $(a_1+a_2,b_1+b_2) \in C$ and $(a_1a_2,b_1b_2) \in C$. Then $C$ is a congruence.
\end{lemma}
\begin{proof}
We only need to show that $C$ is transitive. Assume that $(a,b),(b,c) \in C$. If $b = 0$, then $(a+0, 0+c) = (a,c) \in C$. If $b \neq 0$ then $(b^{-1},b^{-1})\in C$ and $(ab, bc) \in C$, and after multiplying it follows that $(a,c) \in C$.
\end{proof}

\begin{proposition}\label{prop:fracrcongs}
Let $R$ be a cancellative semiring. For a congruence $C$ of $R$ denote by $\<C\>_{\fracr}$ the congruence generated by $C$ in $\fracr$.
\begin{itemize}
\item[(i)] $(a,b) \in \<C\>_{\fracr}$ if and only if there is an $s \in R\setminus\{0\}$ such that $(sa,sb) \in C$. In particular $\<C\>_{\fracr}$ is proper if and only if $\ker(C) = \{0\}$.
\item[(ii)] If $C$ is a QC congruence of $R$ with $\ker(C)= \{0\}$ then $\<C\>_{\fracr}|_R = C$ and for any congruence $\otherc$ of $\fracr$ we have $\<\otherc|_R\>_{\fracr} = \otherc$. 
\item[(iii)] If $C$ is a QC congruence of $R$ with $\ker(C) =\{0\}$, then $C$ is prime if and only if $\<C\>_{\fracr}$ is prime. If $\otherc$ is a congruence of $\fracr$ then $\otherc$ is prime if and only if $\otherc|_R$ is prime.
\end{itemize}
\end{proposition}
\begin{proof}
For (i) set $$C' = \{(a,b) \in \fracr\times\fracr|\:\exists s\in R\setminus\{0\}:\: (sa,sb)\in C\}.$$ Since every $s\in R\setminus\{0\}$ has a multiplicative inverse in $\fracr$ it is clear that $C \subseteq C' \subseteq \<C\>_{\fracr}$. Hence one only needs to see that $C'$ is a congruence. If $s_1,s_2 \in R\setminus\{0\}$ is such that $(s_1 a_1,s_1 b_1) \in C$ and  $(s_2 a_2,s_2 b_2) \in C$ for some $(a_1,b_1),(a_2,b_2) \in  \fracr\times\fracr$ then we have $$(s_1 s_2(a_1+a_2), s_1 s_2(b_1+b_2)) \in C$$ and  $$(s_1 s_2(a_1 a_2), s_1 s_2(b_1 b_2)) \in C$$ showing that $C'$ is closed under addition and multiplication (note that $s_1s_2 \neq 0$ since $R$ is cancellative). Since $C'$ is clearly symmetric and reflexive it follows from Lemma \ref{lem:smftransitive} that $C'$ is indeed a congruence. It follows that $\<C\>_{\fracr}$ is proper if and only if there exists no $s  \in  R\setminus\{0\}$ such that $(s,0) \in C$ or equivalently if $\ker(C) = \{0\}$. \par\smallskip
For (ii) first note that it is immediate from the definition of $C'$ that if $C$ is a QC congruence of $R$ with $\ker(C)= \{0\}$ then $C' \cap R \times R = C$, implying that $\<C\>_{\fracr}|_R = C$. On the other hand if $\otherc$ is a congruence of $\fracr$ then it is clear that  $\<\otherc|_R\>_{\fracr} \subseteq \otherc$. For the other direction if $(\frac{r_1}{s_1}, \frac{r_2}{s_2}) \in\otherc$ then $(r_1s_2, r_2s_1) \in \otherc|_R$ implying that  $(\frac{r_1}{s_1}, \frac{r_2}{s_2}) \in \<\otherc|_R\>_{\fracr}$. \par\smallskip
For the first statement of (iii) recall that the restriction of a prime to a subsemiring is always a prime, hence  
if $\<C\>_{\fracr}$ is a prime congruence, where $C$ is a congruence of $R$ with $\ker(C) =\{0\}$, then $C =\<C\>_{\fracr}|_R$ is also a prime. For the other direction assume that $C$ is a prime of $R$ with  $\ker(C) =\{0\}$ and  we have a twisted product $(\frac{r_1}{s_1}, \frac{r_2}{s_2})(\frac{r'_1}{s'_1}, \frac{r'_2}{s'_2}) \in \<C\>_{\fracr}$. Then by (i) it follows that $(r_1s_2, r_2s_1)(r'_1s'_2, r'_2s'_1) \in C$. Since $C$ is a prime congruence we obtain that one of the factors in the twisted product, say $(r_1s_2, r_2s_1)$, has to be in $C$ and thus $(\frac{r_1}{s_1}, \frac{r_2}{s_2}) \in \<C\>_{\fracr}$ showing that $\<C\>_{\fracr}$ is prime. The second statement in (iii) follows from the first statement and (ii).
\end{proof}

We also recall the following well-known statement:

\begin{proposition}\label{prop:smfcong1}
In a semifield every proper congruence is determined by the equivalence class of $1$.
\end{proposition}
\begin{proof}
Indeed if $C$ is a proper congruence of a semifield then $\ker(C)= \{0\}$ and $(a,b) \in C$ if and only if $a=b=0$ or $(a/b, 1) \in C$.
\end{proof}

Next we collect some elementary observations about additively idempotent semifields that are domains which we will need to prove our main result. We point out that an additively idempotent semifield needs not to be a domain in general. If $A$ is a cancellative $\mb$-algebra that is not totally ordered  (see \cite{JM14} for several such examples)  then by Proposition \ref{prop:primeorder} $\fraca$ is an additively idempotent semifield that is not a domain. In the proof of Proposition \ref{prop:smfdomains} we will often use the following trivial but important fact:

\begin{lemma}
Let $A$ be a $\mb$-algebra. If $x,y \in A$ both have multiplicative inverses then $x \geq y$ if and only if $1/y \geq 1/x$.
\end{lemma}
\begin{proof}
$x \geq y$ means $x+y = x$, multiplying both sides by $\frac{1}{xy}$ we get $1/y+1/x= 1/y$ showing that $1/y \geq 1/x$.
\end{proof}

\begin{proposition}\label{prop:smfdomains}
Let $F$ be an additively idempotent semifield that is a domain.
\begin{itemize}
\item[(i)] Every proper congruence of $F$ is prime. 
\item[(ii)] The congruences of $F$ form a chain. Moreover if $\dim F$ is finite, then every congruence is principal, i.e. generated by $(1,x)$ for some $x \in F\setminus\{0\}$.
\item[(iii)] For $x,y \in  F\setminus\{0\}$, we have that $(1,y) \in \<(1,x)\>$ if and only if there exist an $n \in \mz$ such that $x^{-n} \leq y \leq x^n$.
\end{itemize}
\end{proposition}
\begin{proof}
First note that a proper congruence of any semifield is always cancellative, since if $(ca,cb) \in C$ for $c \neq 0$ then multiplying by $c^{-1}$ we get $(a,b) \in C$. Now (i) follows from Proposition \ref{prop:primeorder} and the fact that the quotient of a totally ordered $\mb$-algebra is also totally ordered. \par\smallskip
For (ii) assume that there are two congruences $C_1$ and $C_2$ such that $C_1 \not\subseteq C_2$ and $C_2 \not\subseteq C_1$.  Then by Proposition \ref{prop:smfcong1} we have $x,y \in F\setminus\{0\}$ such that $(1,x) \in C_1\setminus C_2$ and  $(1,y) \in C_2\setminus C_1$. By possibly replacing $x$ or $y$ with their multiplicative inverse we may assume that $x,y \geq 1$. Moreover $F$ is totally ordered, thus without loss of generality we can set $x \geq y$. Now it follows from (ii) of Proposition \ref{prop: congbasic} that $(1,y) \in C_1$, a contradiction. When $\dim F$ is finite this implies that there is a unique chain of primes $\diag = P_0 \subset P_1 \dots\subset P_{\dim F}$ in F. Choosing any $(a,b) \in P_k \setminus P_{k-1}$ we see that $\<(a/b,1)\> = P_k$ proving the second statement in (ii). \par\smallskip
For (iii) let $H \subset F \times F$ be the set that consists of the pair $(0,0)$ and the pairs $(a,b) \in  (F\setminus\{0\})\times  F\setminus\{0\}$ for which exists an $n \in \mz$ such that $x^{-n} \leq b/a \leq x^n$. We need to show that $H = \<(1,x)\>$ to prove the claim. Clearly we have $(1,x) \in H$ and by Proposition \ref{prop:primeorder} we also have that $H \subseteq \<(1,x)\>$ so we only need to show that $H$ is a congruence. Let $(a_1,b_1),(a_2,b_2) \in H$ and let $n_1,n_2$ be integers such that  $x^{-n_1} \leq b_1/a_1 \leq x^{n_1}$ and  $x^{-n_2} \leq b_2/a_2 \leq x^{n_2}$. We can replace the one of $n_{1,2}$ with the smaller absolute value by the other and assume that $n_1 = n_2 = n$. Now we have that $$x^{-n} = \frac{x^{-n}b_1+x^{-n}b_2}{b_1+b_2}\leq\frac{a_1+a_2}{b_1+b_2}\leq\frac{x^{n}b_1+x^{n}b_2}{b_1+b_2}= x^n,$$ showing that $(a_1+a_2,b_1+b_2) \in H$. To show that $H$ is closed under products, consider the inequalities:  $$x^{-2n} = \frac{(x^{-n}b_1)(x^{-n}b_2)}{b_1b_2}\leq\frac{a_1a_2}{b_1b_2}\leq\frac{(x^{n}b_1)(x^{n}b_2)}{b_1b_2}= x^{2n},$$ implying that $(a_1a_2, b_1b_2) \in H$. Finally $H$ is symmetric since $x^{-n} \leq b/a \leq x^n$ if and only if $x^{-n} \leq a/b \leq x^{n}$, hence by Lemma \ref{lem:smftransitive} $H$ is a congruence.
\end{proof}

\begin{corollary}
If an $\mb$-algebra $A$ is a domain, then the prime congruences of $A$ with trivial kernels form a chain.
\end{corollary}
\begin{proof}
This follows immediately from Proposition \ref{prop:fracrcongs} and (ii) of Proposition \ref{prop:smfdomains}.
\end{proof}

\begin{remark}{\rm
It can be read off from the proof that (iii) of Proposition \ref{prop:smfdomains} holds for any additively idempotent semifield. We would also like to point out that the statement of (iii) is likely well known.
}
\end{remark}

Let $F$ be an additively idempotent semifield that is a domain. For $x,y \in F\setminus\{0\}$ we will write $x \dsim_F y$ whenever $\<(1,x)\>_F = \<(1,y)\>_F$. It follows from (i) and (ii) of Proposition \ref{prop:smfdomains} that when $F$ is finite dimensional the number of $\dsim_F$ equivalence classes  is $\dim F + 1$.

\begin{lemma}\label{lem:stickykernels}
Let $A$ be a $\mb$-algebra that is a domain, and $x,y,z \in A\setminus\{0\}$ with $(1,x) \in \<(1,\frac{y}{z})\>_{\fraca}$. Then for any prime congruence $P$ with $x \in \ker(P)$ we also have that at least one of $y \in \ker(P)$ or $z \in \ker(P)$ hold.
\end{lemma}
\begin{proof}
By (iii) of Proposition \ref{prop:smfdomains} we have that there exist an $n \in \mz$ such that $\frac{z^{n}}{y^{n}} \leq x \leq \frac{y^n}{z^n}$. If $n \geq 0$ then after multiplying by $y^{n}$ we obtain $z^n \leq xy^n$. Since $xy^n \in \ker(P)$ by Proposition \ref{prop: congbasic} we have that $z^n \in ker(P)$. Since $P$ is prime it follows that $z \in \ker(P)$. If $n < 0$ then after multiplying by $z^{-n}$ we obtain that $y^{-n} \leq xz^{-n}$. Since $xz^{-n} \in A$ we have $xz^{-n} \in \ker(P)$ and it follows that $y^{-n} \in \ker(P)$ and thus $y \in \ker(P)$.
\end{proof}

\begin{proposition}\label{prop:trivkerchain}
Let $A$ be a $\mb$-algebra that is a domain, with $\dim A < \infty$. Then $\dim A = \dim \fraca$, in particular the primes of $A$ with a trivial kernel form a chain of maximal length.
\end{proposition}
\begin{proof}
First it follows immediately from Proposition \ref{prop:fracrcongs} that $\dim A \geq \dim\fraca$ since the unique chain of primes in $\fraca$ restricts to a chain of primes in $\dim A$ of the same length.
We will prove by induction on $\dim \fraca$. If $\dim \fraca = 0$ then by Proposition \ref{prop: dimB} $\fraca \simeq \mb$, and since $A$ embeds into $\fraca$ we also have that $A \simeq \mb$. \par\smallskip
Next we assume that $\dim \fraca = d > 0$ and that the claim holds for all $d' < d$. Let $\diag = P_0 \subset P_1 \subset \dots \subset P_{\dim A}$ be a chain of maximal length in $A$ and set $A' = A/P_1$. Clearly $\dim A' = \dim A - 1$. If $\ker(P_1) = \{0\}$ then applying Proposition \ref{prop:fracrcongs} we see that $P_1$ extends to a prime $\<P_1\>_{\fraca}$ of $\fraca$ and $\dim \fraca/\<P_1\>_{\fraca} = d-1$. It follows that $\dim \fracap = d-1$ and applying the induction hypothesis we obtain $\dim A' = d-1$, and thus $\dim A = d$. \par\smallskip
We are left to deal with the case when $0 \neq x \in \ker(P_1)$. First note that the elements of $\fracap$ can be written as $\frac{[a]}{[b]}$ with $a,b \in A$ and $b \notin \ker(P_1)$, where $[a],[b]$ denote the images of $a,b$ in $A'$. (Note however that there is no natural map from $\fraca$ to $\fracap$ in this case.) Now it follows from (iii) of Proposition \ref{prop:smfdomains} that for $\frac{[a]}{[b]}, \frac{[c]}{[d]} \in \fracap$, if we have that $\frac{a}{b} \dsim_{\fraca} \frac{c}{d}$ then $\frac{[a]}{[b]} \dsim_{\fracap} \frac{[c]}{[d]}$. 
Finally it follows from Lemma \ref{lem:stickykernels} that whenever $x \dsim_{\fraca} \frac{a}{b}$ at least one of $a$ or $b$ map to $0$ in $A'$, hence $\dsim_{\fracap}$ has strictly less equivalence classes than $\dsim_{\fraca}$. We obtained that $\dim\fracap \leq d-1$, and hence by the induction hypothesis we have that $\dim A' = \dim\fracap$ and it follows that $\dim A = \dim A'+1 = d$.
\end{proof}

We are ready to state our main result:

\begin{theorem}\label{thm:dplusone}
Let $A$ be a $\mb$-algebra with $\dim A < \infty$. Then we have that $\dim A(x) = \dim A[x] = \dim A + 1$.
\end{theorem}
\begin{proof}
Let $P_0 \subset P_1 \dots \subset P_{\dim A(x)}$ be a chain of primes of maximal length in $A(x)$. By Proposition \ref{prop:trivkerchain} we may assume that the congruences $P_i/P_0$ have trivial kernel in $A(x)/P_0$ or equivalently that $\ker(P_0) = \ker(P_i)$ for all $0 \leq i \leq \dim A(x)$. Now it follows from (i) of Proposition \ref{prop:laurentdim} that after restricting the chain to $A$, in $P_0|_A \subseteq P_1|_A \subseteq \dots$ equality occurs at most once proving that $\dim A +1 \geq \dim A(x)$. Finally by Proposition \ref{prop: var_inc_dim} we also have that $\dim A+1 \leq \dim A(x)$, proving that $\dim A(x) = \dim A + 1$. The equality $\dim A[x] = \dim A + 1$ can be verified by the same argument.  
\end{proof}

\bibliographystyle{alpha}

\begin{thebibliography}{10}

    \bibitem[CC13]{CC13}
	A. Connes and C. Consani,
	\emph{Projective geometry in characteristic one and the epicyclic category}, Nagoya Mathematical Journal 217 (2015), 95-132.
    
    
    \bibitem[Ei95]{Ei95}
    D. Eisenbud,
    \emph{Commutative algebra: with a view toward algebraic geometry}, Graduate Texts in Mathematics (1995), Springer-Verlag, volume 150.

    \bibitem[Go99]{Go99}
    J. S.  Golan,
    \emph{Semirings  and  Their  Applications}, Kluwer,  Dordrecht (1999)
    
    
    \bibitem[JM14]{JM14}
    D. Jo\'o and K. Mincheva,
    \emph{Prime congruences of idempotent semirings and a Nullstellensatz for tropical polynomials}, arXiv:1408.3817
    

    \bibitem[PR15]{PR15} 
    T. Perri and L. Rowen,
    \emph{Kernels in tropical geometry and a Jordan-H\"{o}lder Theorem}, arXiv:1405.0115
    
    \bibitem[Se54]{Se54}
    A. Seidenberg,
    \emph{On the dimension theory of rings. II.}, Pacific J. Math., Volume 4, Number 4 (1954), 603-614.
  


\end{thebibliography}

\end{document}